\documentclass[12pt,reqno]{article}

\usepackage[usenames]{color}
\usepackage{amssymb}
\usepackage{amsmath}
\usepackage{amsthm}
\usepackage{amsfonts}
\usepackage{amscd}
\usepackage{graphicx}

\usepackage[colorlinks=true,
linkcolor=webgreen,
filecolor=webbrown,
citecolor=webgreen]{hyperref}

\definecolor{webgreen}{rgb}{0,.5,0}
\definecolor{webbrown}{rgb}{.6,0,0}

\usepackage{color}
\usepackage{fullpage}
\usepackage{float}

\usepackage{graphics}
\usepackage{latexsym}
\usepackage{epsf}
\usepackage{breakurl}

\def\Enn{\mathbb{N}}

\def\andd{\, \wedge\, }
\def\orr{\, \vee \, }
\def\TM{Thue--Morse }

\newcommand{\seqnum}[1]{\href{https://oeis.org/#1}{\rm \underline{#1}}}

\DeclareMathOperator{\ce}{\mathsf{cexp}}
\DeclareMathOperator{\per}{\mathsf{per}}
\DeclareMathOperator{\hone}{h0001}
\DeclareMathOperator{\testone}{test0001}
\DeclareMathOperator{\checkone}{check0001}
\DeclareMathOperator{\ndmodk}{ndmodk}
\renewcommand{\exp}{\mathsf{exp}}

\def\Reals{\mathbb{R}}
\def\hatr{\hat{\Reals}}
\def\modd#1 #2{#1\ \mbox{\rm (mod}\ #2\mbox{\rm )}}
\def\overli{\overline}
\def\cA{\mathcal{A}}

\begin{document}

\theoremstyle{plain}
\newtheorem{theorem}{Theorem}
\newtheorem{corollary}[theorem]{Corollary}
\newtheorem{lemma}[theorem]{Lemma}
\newtheorem{proposition}[theorem]{Proposition}
\newtheorem{observation}[theorem]{Observation}

\theoremstyle{definition}
\newtheorem{definition}[theorem]{Definition}
\newtheorem{example}[theorem]{Example}
\newtheorem{conjecture}[theorem]{Conjecture}

\theoremstyle{remark}
\newtheorem{remark}[theorem]{Remark}

\title{Power-free Complementary Binary Morphisms
}

\author{
Jeffrey Shallit\\
School of Computer Science \\
University of Waterloo \\
Waterloo, ON  N2L 3G1 \\
Canada\\
\href{mailto:shallit@uwaterloo.ca}{\tt shallit@uwaterloo.ca} \\
\and 
Arseny Shur \\
Department of Computer Science\\
Bar-Ilan University \\
Ramat Gan 5290002\\
Israel\\
\href{mailto:shur@datalab.cs.biu.ac.il}{\tt shur@datalab.cs.biu.ac.il} \\
\and 
Stefan Zorcic \\
School of Computer Science \\
University of Waterloo \\
Waterloo, ON  N2L 3G1 \\
Canada\\
\href{mailto:szorcic@uwaterloo.ca}{\tt szorcic@uwaterloo.ca}
}

\maketitle

\begin{abstract}
We revisit the topic of power-free morphisms, focusing on the properties of the class of complementary morphisms.
Such morphisms are defined over a $2$-letter alphabet, and map the letters 0 and 1 to complementary words.
We prove that every prefix of the famous \TM word $\mathbf{t}$ gives a complementary morphism that is $3^+$-free and hence $\alpha$-free for every real number $\alpha>3$.
We also describe, using a 4-state binary finite automaton, the lengths of all prefixes of $\mathbf{t}$ that give cubefree complementary morphisms.
Next, we show that $3$-free (cubefree) complementary morphisms of length $k$ exist for all $k\not\in \{3,6\}$. 
Moreover, if $k$ is not of the form $3\cdot2^n$, then the images of letters can be chosen to be factors of $\mathbf{t}$.
Finally, we observe that each cubefree complementary morphism is also $\alpha$-free for some $\alpha<3$; in contrast, no binary morphism that maps each letter to a word of length 3 (resp., a word of length 6) is $\alpha$-free for any $\alpha<3$.

In addition to more traditional techniques of combinatorics on words, we also rely on the Walnut theorem-prover.
Its use and limitations are discussed.
\end{abstract}

\section{Introduction}

\subsection{Basic definitions}

This paper concerns avoidance of certain kinds of repetitions in binary words, and, more specifically, morphisms that preserve  avoidance properties.  We start with some definitions.

Let $w$ be a (finite or infinite) word.  We index words starting at
position $0$, so $w[0]$ is the first letter of $w$.  We indicate a contiguous sub-block within $w$ (also called a {\it factor}) with the notation
$w[i..j] := w[i] \, w[i+1] \cdots w[j]$.
A factor $x$ of $w$ that is neither a prefix nor a suffix of $w$ is called {\it internal}. 
We also say that a factor $w[i..j]$ \emph{occurs at position $i$}.

We consider words over the alphabet $\{0,1\}$.
If $a\in\{0,1\}$, the other letter is called the \emph{complement} of $a$ and denoted by $\overli{a}$.
A word $v[0..n]$ is the \emph{complement} of $w[0..n]$, denoted by $v=\overline{w}$, if $v[i]=\overline{w[i]}$ for $i=0,\ldots,n$.

The length of a finite word $w$ is denoted by $|w|$.
Suppose $w$ is a finite nonempty word of length $n$.
We say that an integer $p$, $1 \leq p \leq n$, is a {\it period\/} of
$w$ if $w[i] = w[i+p]$ for $0 \leq i < n-p$. 
A period is {\it proper\/} if $p<n$.  The smallest period
of $w$ is denoted by $\per(w)$.   The {\it exponent\/} of a word
$w$ is defined to be $\exp(w) = n/\per(w)$, and is a measure of
the repetitivity of the word.  For example, $\exp({\tt entente}) = 7/3$.

If a word has length $2n$ and period $n$, such as
{\tt murmur}, it is called a {\it square}.
If a word has length $2n+1$ and period $n$, such as
{\tt alfalfa}, it is called an {\it overlap}.
If a word has length $3n$ and period $n$, such as
{\tt shshsh}, it is called a {\it cube}.

The {\it critical exponent\/} of a word $w$, denoted $\ce(w)$, is the
supremum, over all nonempty factors $x$ of $w$, of $\exp(x)$.   
For example, the critical exponent of Aristophanes' famous coined word
{\tt brekekekex} is $7/2$.

Let $\alpha>1$ be a real number.   We now define the concept of
$\alpha$-freeness for words:
\begin{itemize}
    \item If $\ce(w) < \alpha$, we say that
$w$ is {\it $\alpha$-free}. 
  \item If $\ce(w) \leq \alpha$, we say that $w$
is {\it $\alpha^+$-free}.  
\end{itemize}
For example, the name of the Dutch mathematician {\tt Lekkerkerker} 
is $3^+$-free, but not $3$-free.
The terms \emph{squarefree}, \emph{overlap-free}, and \emph{cubefree} are commonly used substitutes for $2$-free, $2^+$-free, and $3$-free, respectively.  

It is convenient to extend the real numbers in such a way that every real number $\alpha$
has a partner, $\alpha^+$ that satisfies the inequality $\alpha < \alpha^+ < \beta$ for all $\alpha < \beta$.  This new domain is denoted $\hatr$.  
Somewhat unintuitively, we define $\lceil \alpha^+ \rceil = \lceil \alpha \rceil$ for
$\alpha \in \Reals$.

\subsection{Power-free morphisms}

Recall that a \emph{morphism} $h$ is a map obeying the identity
$h(xy) = h(x) h(y)$ for all words $x, y$. 
We write $h^n$ for the composition $h\circ h\circ\cdots\circ h$  ($n$ times); $h^0$ is the identity map.
An infinite word $\mathbf{u}$ satisfying $h(\mathbf{u})=\mathbf{u}$ is the \emph{fixed point} of $h$.
If a morphism maps each letter to a word of the same length, say length $k$, we call it \emph{($k$-)uniform}; we refer to $k$ as the \emph{length} of such a morphism.
A uniform binary morphism $h$ satisfying $h(1)=\overline{h(0)}$ is called \emph{complementary}.

Let $\alpha \in \hatr$ and $\alpha>1$.   There is also a concept of
$\alpha$-freeness for morphisms, as follows:
 if for all words $w$ we have \\
 \centerline{
 $w$ $\alpha$-free implies that
    $h(w)$ is $\alpha$-free,}\\[.1in]
    then we call $h$ $\alpha$-free.
In other words, an $\alpha$-free morphism is one that preserves the property of being $\alpha$-free.
The fixed point of any $\alpha$-free morphism is $\alpha$-free; however, the converse implication is not always true.

The first result in the theory of $\alpha$-free morphisms was proven by Thue in 1912 \cite{Thue:1912}:
\begin{theorem}
    Define $\mu: 0\mapsto 01, 1\mapsto 10$.  Then $\mu$ is a $2^+$-free morphism.
\end{theorem}
In fact, even more is true.  We have
the following result
\cite{Shur:2000,Karhumaki&Shallit:2004}:

\begin{theorem} \label{t:muexponents}
Let $\alpha \in \hat{\Reals}$ and $\alpha > 2$.  Then $\mu$ is an $\alpha$-free morphism.
\end{theorem}

The study of $\alpha$-free morphisms has been quite active.
Most of the results are related to integer values of $\alpha$ and the ``nearly integer'' value $2^+$.
Thue \cite{Thue:1912} gave the first easily verifiable sufficient condition for a morphism $h$ to be squarefree. The best criterion known for squarefree morphisms is due to Crochemore \cite{Crochemore:1982}.
A similar criterion was found for overlap-free morphisms \cite{Richomme&Wlazinski:2004}.

The case of $\alpha$-free morphisms is more subtle when $\alpha$ is an integer greater than 2. There is a generalization, due to Bean et al.\  \cite{Bean&Ehrenfeucht&McNulty:1979}, of Thue's sufficient condition, but on the other hand there is evidence that no general criteria exist even for $\alpha=3$ \cite{Richomme&Wlazinski:2002a}.
However, such criteria {\it do\/} exist under certain restrictions on morphisms---namely, for \emph{uniform} morphisms \cite{Keranen:1984, Richomme&Wlazinski:2007} and for \emph{binary} morphisms \cite{Wlazinski:2001, Richomme&Wlazinski:2002a}.

The sets of $\alpha$-free morphisms for different integers $\alpha$ satisfy some containment relations.
Thus, every squarefree morphism satisfying a certain additional condition is $\alpha$-free for all integers $\alpha>2$ \cite{Bean&Ehrenfeucht&McNulty:1979}; an even simpler condition was presented in \cite{Leconte:1985}.
Further, every $\alpha$-free \emph{uniform} morphism is $(\alpha+1)$-free whenever $\alpha\ge3$ is an integer \cite{Wlazinski:2016, Wlazinski:2017}. 
On the other hand, for every integer $\alpha\ge3$ there exists an $(\alpha+1)$-free \emph{binary} morphism that is not $\alpha$-free but has an $\alpha$-free fixed point \cite{Richomme&Seebold:2002b}.

The case of non-integer $\alpha$ is less studied. Brandenburg \cite{Brandenburg:1983} analyzed whether a morphism over $k$ letters can be $\alpha_k$-free, where $\alpha_k$ is the minimum critical exponent of an infinite $k$-ary word.
He proved that $\alpha_k$-free morphisms do not exist if $k\ge 4$; for $k=2$ and $k=3$ such morphisms do exist \cite{Thue:1912,Dejean:1972}.
Note that for non-integer $\alpha$, the uniformity condition on morphisms is not restrictive, due to the following observation:
\begin{observation} \label{o:necessary}[\cite{Brandenburg:1983,Kobayashi:1986}]
    If a morphism $h$ is $\alpha$-free for a non-integer $\alpha\in\hatr$, then $h$ is uniform, and for every pair of distinct letters $a,b$ from the domain of $h$, the words $h(a)$ and $h(b)$ have distinct first letters and distinct last letters.
\end{observation}
Kobayashi \cite{Kobayashi:1986} proved the following sufficient condition for morphisms to be $\alpha$-free. 

\begin{theorem} \label{t:kobayashi}
Let $\alpha > 2$ be a member of $\hatr$. Suppose that a $k$-uniform morphism $h$ satisfies the following conditions:
\begin{enumerate}
    \item for every letter $a$ we have $\exp(h(a))<2$;
    \item for every two distinct letters $a,b$, \begin{itemize} 
    \item[(b1)] the words $h(a)$ and $h(b)$ have distinct first letters;
    \item[(b2)] the words $h(a)$ and $h(b)$ have distinct last letters;
    \item[(b3)] if $h(a)=uv$, $h(b)=v'u$, and $|u|\ge k/2$, then $v=v'$;
    \end{itemize}
    \item the word $h(x)$ is $\alpha$-free for each $\alpha$-free word $x$ of length $\leq \lceil \alpha \rceil + 1 $.
\end{enumerate}
Then $h$ is $\alpha$-free.
\end{theorem}

The only case in which the full description of the set of $\alpha$-free morphisms is known is presented in the following theorem. The particular case of overlap-free morphisms was explicitly stated in \cite{Berstel&Seebold:1993}, and the general case easily follows from Theorem~\ref{t:muexponents} and known structural results on $(7/3)$-free words; see, e.g., \cite{Karhumaki&Shallit:2004,Shur:2000}.

\begin{theorem}\label{t:73free}
    Let $\alpha\in\hatr$ and $2^+\le\alpha\le 7/3$.
    A binary morphism is $\alpha$-free if and only if it equals $\mu^n$ or $\theta\circ\mu^n$, where $\theta: 0\mapsto 1, 1\mapsto 0$, $\mu: 0\mapsto 01, 1\mapsto 10$, and $n\ge0$.
\end{theorem}

In particular, Theorem~\ref{t:73free} implies that for $2^+\le\alpha\le 7/3$ a $k$-uniform $\alpha$-free binary morphism exists if and only if $k$ is a power of 2.
However, for $\alpha=3$ the situation is quite different. Currie and Rampersad \cite{Currie&Rampersad:2009}
exhibit a set of morphisms proving the following  result.
\begin{theorem} \label{t:3free_lengths}
    For every $k \geq 1$ there is a $k$-uniform cubefree binary morphism.
\end{theorem}

%

\subsection{Questions and Results}

For the present study, we took Theorem~\ref{t:3free_lengths} as a starting point and addressed a more general question: for arbitrary $\alpha\in \hatr$, what are the possible lengths of uniform $\alpha$-free binary morphisms? 
Specifically, in this paper we answer the following particular questions:
\begin{itemize}
    \item Does the result of Theorem~\ref{t:3free_lengths} extend to arbitrary values $\alpha\in\hatr$ that are greater than 3? If so, does this extended result holds if we are further restricted to the class of complementary morphisms or its subclasses?
    \item Does the result of Theorem~\ref{t:3free_lengths} hold for the class of complementary morphisms or for some its subclasses?
    \item Does the result of Theorem~\ref{t:3free_lengths} hold for any $\alpha<3$?
\end{itemize}

Note that to deal with non-integer powers $\alpha$ we need morphisms satisfying Observation~\ref{o:necessary}, so the class of complementary morphisms is arguably the most natural choice.
We also note that Thue's morphism $\mu$, as well as all other morphisms from Theorem~\ref{t:73free}, are complementary.

\medskip
Our main results are as follows. We start with the first question.

\begin{theorem} \label{t:3plusfree}
    For every $\alpha\ge 3^+$ and each $k \ge 1$, there is an $\alpha$-free complementary morphism of length $k$.
\end{theorem}

To prove Theorem~\ref{t:3plusfree}, we use the infinite \TM word
$$
\mathbf{ t} = 0110100110010110\cdots,
$$ 
which is the fixed point of Thue's morphism $\mu$ defined previously, as a source of images of letters for our morphisms.
Namely, we prove that for each $k\ge 1$, the complementary morphism 
$$h_k: 0\mapsto \mathbf{t}[0..k-1],\ 1\mapsto \overline{\mathbf{t}[0..k-1]}$$ is $\alpha$-free for all $\alpha>3$ (see Theorem~\ref{t:exponents} in Section~\ref{s:trunc}).
Furthermore, we observe that for some values of $k$, the morphism $h_k$ is cubefree, and we characterize these $k$ with a 4-state binary deterministic finite automaton (see Theorem~\ref{t:automaton} in Section~\ref{s:walnut3plus}).

The answer to the second question is given by the following theorem.

\begin{theorem} \label{t:comp3free}
    Cubefree complementary morphisms of length $k$ exist for all values $k \ge 1$ except $k=3$ and $k=6$.
\end{theorem}

Proving Theorem~\ref{t:comp3free}, we find factors of the \TM word that give cubefree complementary morphisms of all lengths $k$ except those of the form $k=3\cdot 2^n$ (see Theorem~\ref{t:TMfact3free} in Section~\ref{s:TMfactors}).
The remaining cases are covered by morphisms of the form $g \circ\mu^n$, where $g$ is a certain complementary morphism of length 12.
Covering the exceptional lengths 3 and 6 with uniform (but not complementary) morphisms $\rho$ and $\rho\circ\mu$, where $\rho: 0\mapsto 001, 1\mapsto 011$, we get an alternative proof of Theorem~\ref{t:3free_lengths}.

In addition, we present two more series of cubefree complementary morphisms; the images of letters in these series are certain prefixes of two infinite words densely related to the \TM word.
Using these two series and the morphisms $\mu$ and $\rho$, one obtains another alternative proof of Theorem~\ref{t:3free_lengths} (see Section 5).

Finally, we slightly extend the argument from Theorem~\ref{t:comp3free} to obtain the following theorem, providing an answer to the third of studied questions.
\begin{theorem} \label{t:below3}
\leavevmode
    \begin{enumerate}
        \item For all $\alpha\in\hatr$ such that $\alpha<3$, $3$-uniform and $6$-uniform $\alpha$-free binary morphisms do not exist.
        \item For each $k\in \mathbb{N}$ such that $k\notin\{3,6\}$ there is a complementary morphism of length $k$ that is $\alpha$-free for some $\alpha<3$.
    \end{enumerate}
\end{theorem}

Along with more traditional methods of combinatorics on words, we use the theorem-prover
\texttt{Walnut} \cite{Mousavi:2016,Shallit:2022} designed to prove first-order sentences related to certain classes of infinite words.
Although this second approach is less general, it is conceptually quite simple, and the proof takes advantage of some existing software to replace case analysis.
It does, however, require some serious computation: sometimes \texttt{Walnut} cannot prove a first-order statement because it exceeds the limits of computational resources in the process. 
Using \texttt{Walnut}, we give an alternative proof of the crucial case $\alpha=3^+$ of Theorem~\ref{t:3plusfree}, though we were unable to prove Theorem~\ref{t:exponents} in this way.
The automaton in Theorem~\ref{t:automaton} is also obtained by \texttt{Walnut}.
The proof of Theorem~\ref{t:TMfact3free} using \texttt{Walnut} currently seems to be beyond its computational limits, and thus the prover cannot help with Theorem~\ref{t:comp3free} or Theorem~\ref{t:below3}.
On the positive side, \texttt{Walnut} proved cube-freeness of the morphisms from two infinite sequences, thus providing the second alternative proof of Theorem~\ref{t:3free_lengths}.

The paper is organized as follows. In Section~\ref{s:comb} we give combinatorial proofs of Theorems~\ref{t:3plusfree},~\ref{t:comp3free}, and~\ref{t:below3}.
In Section~\ref{s:walnut} we introduce \texttt{Walnut} and describe all proofs obtained by its use.
We end up with a short discussion in Section~\ref{s:discussion}.

\section{Combinatorial Proofs} \label{s:comb}

\subsection{Truncated \TM morphisms: proof of Theorem~\ref{t:3plusfree}}
\label{s:trunc}


For each integer $k \geq 1$, we define the
\emph{truncated \TM morphism} $h_k$ as the complementary morphism such that 
$h_k(0)=\mathbf{t}[0..k-1]$.
For example, $h_3(0) = 011$ and $h_3(1) = 100$.

The following theorem is a counterpart to Theorem~\ref{t:muexponents} and implies Theorem~\ref{t:3plusfree}. 
\begin{theorem}\label{t:exponents}
Suppose that $w$ is a binary word and $k$ is a positive integer. 
Then $\ce(h_k(w)) = \ce(w)$ if
$\ce(w) \geq 3$ and $\ce(h_k(w))\le 3$ if $\ce(w)<3$.
\end{theorem}

The proof of Theorem~\ref{t:exponents} is based on some properties of prefixes of $\mathbf{t}$ (we also call them \TM prefixes).
These properties are proved below in Lemmas~\ref{l:blocks}--\ref{l:periods}, and may be of independent interest.
We refer to the words $h_k(0)$ and $h_k(1)$ as \emph{$k$-prefixes}  and to $\mu^n(0), \mu^n(1)$ as \emph{$n$-blocks}.
The following auxiliary property was apparently known to Thue~\cite{Thue:1912}.

\begin{lemma}\label{l:blocks}
    Let $a\in\{0,1\}$, $n\ge1$.
    \begin{itemize}
    \item[(a)] In the word $\mu^n(a\,\overli{a})$, the $n$-blocks occur only at positions $0$ and $2^n$.
    \item[(b)] In the word $\mu^n(aa)$, the $n$-blocks occur only at positions $0, 2^{n-1}$, and $2^n$.
    \end{itemize}
\end{lemma}
\begin{proof}
    Since equal factors of an overlap-free word do not overlap, we immediately get (a). 
    For (b), we write $\mu^n(aa)=\mu^{n-1}(a \, \overli{a} \, a \, \overli{a})$ to see the occurrence of the $n$-block $\mu^n(\overli{a})$ at position $2^{n-1}$ and use (a) to show that there are no other occurrences.
\end{proof}
\begin{lemma}\label{l:intersection}
    Let $k\ge1$. Suppose a word $u$  begins with a $k$-prefix and ends with a $k$-prefix, and $|u| > k$.
    Then $|u|\ge 3k/2$.  Furthermore, equality is attained if and only if $k$ is a power of 2.
\end{lemma}
\begin{proof}
    Let $u$ begin with a $k$-prefix $s$ and end with a $k$-prefix $s'$.   The result is trivial for
    $k=1$, so assume $k>1$.
    
    First consider the case $k=2^n$ for some $n\ge1$.
    We have $s=\mu^n(a)=\mu^{n-1}(a)\mu^{n-1}(\overli{a})$ for some $a\in\{0,1\}$.
    By Lemma~\ref{l:blocks} (a), the word $s$ has occurrences of $(n-1)$-blocks only at the positions 0 and $2^{n-1}$.
    Since $s'$ begins with a $(n-1)$-block, it cannot begin with $u$ at a position smaller than $2^{n-1}$.
    If $s'$ begins at this position, we have $u=\mu^{n-1}(a\overli{a}a)$ and $|u|=3k/2$; otherwise, $|u|>3k/2$.

    Now let $2^n<k<2^{n+1}$ for some $n\ge1$.
    Then $s$ has the prefix $\mu^{n-1}(a\overli{a})\overli{a}$ and is a proper prefix of $\mu^{n-1}(a\,\overli{a}\,\overli{a}\,a)$.
    If $s'$ overlaps $s$ by at most $2^{n-1}$ symbols, then $k>2^n$ implies $|u|>3k/2$.
    Now suppose that $s'$ overlaps $s$ by more than $2^{n-1}$ symbols.
    In particular, $s'$ begins with a $(n-1)$-block contained in $s$.
    Lemma~\ref{l:blocks} gives us all possible positions of this $(n-1)$-block in $s$.
    At position $2^{n-1}$, the word $s$ contains the factor $\mu^{n-1}(\overli{a})\overli{a}$, which is not a $k$-prefix and thus not a prefix of $s'$.
    At position $3\cdot 2^{n-2}$, the word $s$ contains the factor $\mu^{n-1}(a)a$, which is not a $k$-prefix as well.
    Finally, if $s'$ begins at position $2^n$ of $u$, the condition $k<2^{n+1}$ implies that $s'$ overlaps $s$ by less than $k/2$ positions.
    Hence $|u|>3k/2$, as required.    
\end{proof}

\begin{lemma} \label{l:nointernal}
Suppose that $k$ is not a power of 2.
Then the concatenation of two $k$-prefixes does not contain a $k$-prefix as an internal factor.
\end{lemma}
\begin{proof}
    The existence of such a factor contradicts Lemma~\ref{l:intersection}.    
\end{proof}

\begin{lemma}\label{l:periods}
    Let $p=m\cdot 2^n$ with $m$ odd and $n\ge 0$.
    \begin{itemize}
    \item[(a)]  If $m\ge 5$, then any factor of $\mathbf{t}$ with period $p$ has length at most $(m+3)\cdot 2^n$.
    \item[(b)] If some $k$-prefix has proper period $p$, then
        \begin{itemize}
        \item[(i)] $m\ne 1$;
        \item[(ii)] if $m=3$, then $k\le 5\cdot 2^n$; and 
        \item[(iii)] if $m=5$, then $k\le 6\cdot 2^n$.
        \end{itemize}
    \end{itemize}
\end{lemma}
\begin{proof}
\leavevmode
    \begin{itemize}
    \item[(a)]
    Let $u$ be the longest factor of $\mathbf{t}$ with period $p$. 
    Write $u=u_0u_1\cdots u_{r-1}u_r$, where $u_0$ is a nonempty suffix of an $n$-block, $u_1,\ldots, u_{r-1}$ are $n$-blocks, and $u_r$ is a nonempty prefix of an $n$-block.
    If $r\le m+2$, then the length requirement trivially holds, so let us assume $r\ge m+3$.
    Since $u$ has period $p$, the word $u_0$ is a suffix of $u_m$, and $u_r$ is a prefix of $u_{r-m}$.
    Then $u_ru_1\cdots u_{r-1}u_{r-m}$ has period $p$.
    Since $u$ is the longest word with this period, we have $u_0=u_m$ and $u_r=u_{r-m}$.
    Thus $u$ is a product of $n$-blocks.
    Since $m$ is odd, either $u_0u_1$ or $u_mu_{m+1}$ is an $(n+1)$-block.
    As $u_0u_1=u_mu_{m+1}$, we conclude that $u_1\ne u_0$.
    By the same argument, $u_2\ne u_1$ and $u_3\ne u_2$.
    Hence, $u$ begins with $\mu^n(a\,\overli{a}\,a\,\overli{a})$ for some $a\in\{0,1\}$ and has another occurrence of this factor at position $p$.
    Since $\mathbf{t}$ is overlap-free, the factor $\mu^n(a\,\overli{a}\,a\,\overli{a})$ in $\mathbf{t}$ is always followed by $\mu^n(\overli{a})$.
    Thus, $r\ge m+4$ and $u_0\cdots u_4=u_m\cdots u_{m+4}=\mu^n(a\,\overli{a}\,a\,\overli{a}\,\overli{a})$.
    This contradicts the fact that either $u_3u_4$ or $u_{m+3}u_{m+4}$ is an $(n+1)$-block.
    As our assumption $r\ge m+3$ implies contradiction, we get $r\le m+2$ and then $|u|\le (m+3)\cdot 2^n$.

    \item[(b)] Let $a$ be the first letter of the $k$-prefix.
    The $(2^n+1)$-prefix equals $\mu^n(a)\overli{a}$.
    The last $\overli{a}$ immediately breaks the period $2^n$, so we get (i).
    Next, the ($5\cdot2^n+1$)-prefix equals $\mu^n(a\,\overli{a}\,\overli{a}\,a\,\overli{a})a$.
    The last $a$ breaks the period $3\cdot2^n$, implying (ii).
    Finally, the ($6\cdot2^n+1$)-prefix equals $\mu^n(a\,\overli{a}\,\overli{a}\,a\,\overli{a}\,a)a$.
    The last $a$ breaks the period $5\cdot2^n$, implying (iii).
    \end{itemize}
\end{proof}

\begin{proof}[Proof of Theorem~\ref{t:exponents}]
    We suppose that $k$ is not a power of 2, for otherwise the result follows from Theorem~\ref{t:muexponents}.
    
    Let $W=h_k(w)$ and let $p$ be an arbitrary positive integer.
    In $W$, consider the longest factor $U$ with $\per(U)=p$.
    To prove the theorem it is sufficient to show that either $\exp(U)\le 3$ or $w$ has a factor of exponent $\exp(U)$.
    We assume that $\exp(U)> 2$ (otherwise, there is nothing to prove).
    Then $U$ is not a factor of a $k$-prefix and thus can be written as $U=u_0u_1\cdots u_{r-1}u_r$, where $r\ge1$, $u_0$ is a nonempty suffix of a $k$-prefix, $u_1,\ldots,u_{r-1}$ are $k$-prefixes, and $u_r$ is a nonempty prefix of a $k$-prefix.
    Let us consider the possible options for $p$ and $k$.
    We write $p=m\cdot 2^n$ with $m$ odd.
    
    \bigskip\noindent\emph{Case 1}: $p=ck$ for an integer $c$.
    Then $u_0$ is a suffix of $u_c$, $u_r$ is a prefix of $u_{r-c}$, and the word $u_ru_1\cdots u_{r-1}u_{r-c}$ has period $p$ (cf.\ the proof of Lemma~\ref{l:periods} (a)).
    Since $U$ has maximal length among factors of $W$ with period $p$, we have $u_0=u_c$, $u_r=u_{r-c}$, implying that $U$ is a product of $k$-prefixes.
    Then its preimage $u=h_k^{-1}(U)$, which is a factor of $w$, has period $c=p/k$.
    Since $|u|=|U|/k$, we have $\exp(u)=\exp(U)$, as required.

    \bigskip\noindent\emph{Case 2}: $p<k$ and $r=1$. 
    If $p \ge|u_1|$, then $\exp(U)\le 3$, since $u_0$ is overlap-free. 
    So let $p<|u_1|$. 
    Hence $p$ is a proper period of a prefix of $\mathbf{t}$ or $\overli{\mathbf{t}}$. 
    We apply Lemma~\ref{l:periods}.
    Its statements (a) and (b) (iii) imply that $|U|\le(2m+4)\cdot2^n$ for $m=5$ and $|U|\le (2m+6)\cdot2^n$ for $m\ge7$.
    In both cases, $\exp(U)<3$.
    As statement (b) (i) guarantees $m\ne1$, only the case $m=3$ remains.
    
    In this case, $u_1$ is a prefix of $\mu^n(a\,\overli{a}\,\overli{a}\,a\,\overli{a})$ for some $a\in\{0,1\}$ by statement (b) (ii).
    To preserve the period $p$, $u_0$ must be a suffix of some word $\mu^n((a\,\overli{a}\,\overli{a})^i)$.
    The longest such suffix that is also a factor of $\mathbf{t}$ is $\mu^n(\overli{a}\,a\,\overli{a}\,\overli{a})$.
    Indeed, $\mu^n(\overli{a}\,a\,\overli{a}\,\overli{a})$ is preceded in $\mathbf{t}$ by the last letter of $\mu^n(a)$ and not of $\mu^n(\overli{a})$.
    Overall, $|U|=|u_0|+|u_1|\le 9\cdot2^n\le 3p$, and we get $\exp(U)\le 3$, as required.

    \bigskip\noindent\emph{Case 3}: $p<k$ and $r>1$.
    A product of two $k$-prefixes has no period $p<k$: if such a period exists, this product contains a $k$-prefix as an internal factor, contradicting Lemma~\ref{l:nointernal}.
    Hence $r=2$.
    Let $u_2'$ be the $k$-prefix containing $u_2$ as a prefix. 
    If $|u_2|\ge p$, then the product $u_1u_2'$ contains a copy of $u_1$ as an internal factor, again contradicting Lemma~\ref{l:nointernal}.
    Hence $|u_2|<p$.
    A symmetric argument implies $|u_0|<p$; a better bound for $|u_0|$ is proved below.

    Let $u_1=yz$, where $|y|=p$.
    Since $p$ is a proper period of $u_1$, we have $m\ge 3$ by Lemma~\ref{l:periods} (b).
    As $y$ is a $p$-prefix, we can write  $y=s_1\cdots s_m$, where $s_1,\ldots,s_m$ are $n$-blocks.
    Let $s_1=\mu^n(a)$; then $s_2=s_3=\mu^n(\overli{a})$.
    Further, $z$ is a $(k-p)$-prefix, and therefore a prefix of $\mu^n(a\,\overli{a}\,\overli{a})$ by Lemma~\ref{l:periods} (a).
    Therefore, $u_1$ is a prefix of $\mu^n(a\,\overli{a}\,\overli{a})s_4\cdots s_m\mu^n(a\,\overli{a}\,\overli{a})$.
    As $m$ is odd, the factor $s_m\mu^n(a)$ of $\mathbf{t}$ is an $(n+1)$-block, which means $s_m=\mu^n(\overli{a})$. 
    This $(n+1)$-block is followed in $\mathbf{t}$ by another $(n+1)$-block and not by  $\mu^n(\overli{a}\,\overli{a})$.
    Hence, $z$ is a prefix of $\mu^n(a\,\overli{a})$.
    Furthermore, $s_{m-2}s_{m-1}$ and $s_{m-4}s_{m-3}$ (if it exists) are also $(n+1)$-blocks.
    Since $y$ is overlap-free, it therefore ends either with $\mu^n(\overli{a}\,\overli{a})$ or with $\mu^n(\overli{a}\,\overli{a}\,a\,\overli{a})$.
    
    Since $p$ is a period of $U$ and $|u_0|<p$, we know that $u_0$ is a suffix of $y$.
    If $|u_0|<2^n$, we have $|u_0u_1|<2p$ by Lemma~\ref{l:periods} and thus $|U|<3p$, implying $\exp(U)<3$.
    So let $|u_0|\ge 2^n$.
    Then $u_0$ ends with the $n$-block $\mu^n(\overli{a})$.
    As $u_0$ is a suffix of either $u_1$ or $\overli{u_1}$, the word $u_1$ also ends with an $n$-block.
    By Lemma~\ref{l:blocks} (a), only two cases are possible: either $z=\mu^n(a\,\overli{a})$ and $u_0$ is a suffix of $u_1$, or $z=\mu^n(a)$ and $u_0$ is a suffix of $\overli{u_1}$.
    Recall that $u_0$ is also a suffix of $y$.
    
    If $z=\mu^n(a\,\overli{a})$, then $y$ ends with $\mu^n(\overli{a}\,\overli{a})$ (otherwise $u_1$ is not overlap-free).
    Hence the longest common suffix of $y$ and $u_1$ is $\mu^n(\overli{a})$.
    As a result, $|u_0|= 2^n$, $|u_1|= (m+2)\cdot 2^n$, and so $|u_0u_1|\le 2p$.
    If $z=\mu^n(a)$, both options for $y$ are valid.
    If $y$ ends with $\mu^n(\overli{a}\,\overli{a})$, the longest common suffix of $y$ and $\overli{u_1}$ is $\mu^n(\overli{a})$.
    We get $|u_0|= 2^n$, $|u_1|= (m+1)\cdot 2^n$, and so $|u_0u_1|< 2p$.
    If $y$ ends with $\mu^n(\overli{a}\,\overli{a}\,a\,\overli{a})$, the longest common suffix of $y$ and $\overli{u_1}$ is $\mu^n(\overli{a}\,a\,\overli{a})$; besides, we have $m\ge 5$.
    Then $|u_0|\le 3\cdot 2^n$, $|u_1|= (m+1)\cdot 2^n$, and again $|u_0u_1|< 2p$.
    Since $|u_2|<p$, in all cases we have $|U|<3p$ and thus $\exp(U)<3$.

    \bigskip\noindent\emph{Case 4}: $p>k$ and $p$ is not divisible by $k$.
    If the word $U=u_0u_1\cdots u_r$ has length at least $|u_0|+k+p$, then its factor $u_1$ has an occurrence at position $|u_0|+p$.
    Since $p$ is not divisible by $k$, this occurrence is an internal factor of the product of two $k$-prefixes, contradicting Lemma~\ref{l:nointernal}.
    Therefore, $|U|<|u_0|+k+p<3p$ and $\exp(U)<3$, as required.

    Cases 1--4 exhaust all possibilities for $p$. 
    In Case 1, $w$ has a factor with exponent $\exp(U)$, while in Cases 2--4, $\exp(U)\le 3$. 
    Thus, the theorem is proved.
\end{proof}

\begin{proof}[Proof of Theorem~\ref{t:3plusfree}]
    Immediate from Theorem~\ref{t:exponents} and the definition of $\alpha$-free morphism.
\end{proof}

\subsection{Clipped \TM morphisms: proof of Theorem~\ref{t:comp3free}} \label{s:TMfactors}

While Theorem~\ref{t:exponents} implies that all morphisms $h_k$ are $3^+$-free, many of them are not cubefree. 
Indeed, if $h_k(a)$, where $a$ is a letter, ends with a square of a letter or of an $n$-block, then either $h_k(a\,a)$ or $h_k(a\,\overli{a})$ contains the cube of this letter (or of this $n$-block).
However, if instead of prefixes we consider all factors of the \TM word, we can build cubefree complementary morphisms similar to $h_k$ for almost all lengths.

For integers $\ell, k\ge 1$, define the \emph{clipped \TM morphism} $f_{\ell,k}$ as the complementary morphism such that 
$f_{\ell,k}(0)=\mathbf{t}[\ell..\ell+k-1]$.
In particular, $h_k=f_{0,k}$. 
When we are interested in the factor $v=\mathbf{t}[\ell..\ell{+}k{-}1]$ rather than in particular numbers $\ell$ and $k$, we write  $f_v$ instead of $f_{\ell,k}$.
\begin{theorem}\label{t:TMfact3free}
    Let $k=m\cdot 2^n$, where $m$ is odd and $n\ge0$.
    A cubefree clipped \TM morphism of length $k$ exists if and only if $m\ne 3$. Moreover, 
    \begin{enumerate}
        \item if $m\equiv \modd{1} {8}$, then $f_{4\cdot 2^n,k}$ is cubefree;
        \item if $m\equiv \modd{3} {8}$, then either $f_{k-5,k}$ or $f_{k-4,k}$ is cubefree;
        \item if $m\equiv \modd{5} {8}$, then $f_{0,k}$ is cubefree;
        \item if $m\equiv \modd{7} {8}$, then $f_{6\cdot 2^n,k}$ is cubefree.
    \end{enumerate}
\end{theorem}

The proof of Theorem~\ref{t:TMfact3free} is based on two lemmas that establish sufficient conditions for cubefree and non-cubefree clipped \TM morphisms.

\begin{lemma} \label{l:cubefree}
    Let $v$ be a factor of $\mathbf{t}$, having either
    \begin{enumerate}
    \item the prefix $0110$ or $1001$ and the suffix $0010$ or $1101$, or
    \item the prefix $0100$ or $1011$ and the suffix $0110$ or $1001$.
    \end{enumerate}
    Then the morphism $f_v$ is cubefree.
\end{lemma}
\begin{proof}
    As the reversal of an $\alpha$-free word is $\alpha$-free and the reversal of a factor of $\mathbf t$ is a factor of $\mathbf t$, we can assume that  $v$ satisfies $(a)$.
    
    As in the proof of Theorem~\ref{t:exponents}, we take an arbitrary word $w\in\{0,1\}$, consider its image $W=f_v(w)$, fix an arbitrary period $p$ and study the longest factor $U$ of $W$ satisfying $\per(U)=p$.
    We prove the lemma by assuming $\exp(U)\ge 3$ and presenting a factor of $w$ with the exponent $\exp(U)$.
    Since $\ce(v)\le \ce(\mathbf{t})=2$, the word $U$ is not a factor of $v$ (and of $\overli{v}$) and thus can be represented as $U=u_0u_1\cdots u_{r-1}u_r$, where $u_0$ is a nonempty suffix of $v$ or $\overli{v}$, $u_r$ is a nonempty prefix of $v$ or $\overli{v}$, and $u_1\ldots u_{r-1}\in\{v,\overli{v}\}$.
    We directly check that $\exp(U)<3$ for $p=1,2,3,4$, and thus assume $p\ge 5$.

    Let $a$ be the last letter of $v$. 
    We call the words $a\,a\,\overli{a}\,a\,a$, $a\,\overli{a}\,a\,\overli{a}\,a$, and their complements \emph{markers}.
    The markers are not factors of $\mathbf t$, and thus also not factors of $v$, while each of the words $v\,v$, $v\,\overli{v}$, $\overli{v}\,v$, $\overli{v}\,\overli{v}$ contains a unique occurrence of a marker:
    $$
\cdots\pmb{a\,a\,\overli{a}\,a\ a}\,\overli{a}\,\overli{a}\,a\cdots \text{ or } 
\cdots a\,\pmb{a\,\overli{a}\,a\ \overli{a}\,a}\,a\,\overli{a}\cdots \text{ or }
\cdots \pmb{\overli{a}\,\overli{a}\,a\ \overli{a}\,\overli{a}}\,a\,a\,\overli{a}\cdots \text{ or }
\cdots \overli{a}\,\pmb{\overli{a}\,a\ \overli{a}\,a\,\overli{a}}\,\overli{a}\,a\cdots\ .
    $$
    Since $\exp(U)\ge 3$ and $v$ is overlap-free, we have $|U|\ge |v|+p$.
    As $p\ge 5$, either $|u_r|\ge 2$ or $|u_{r-2}|\ge 4$.
    In each case, some factor $u_iu_{i+1}$ of $U$ contains a marker.
    In a cube of period $p$, every factor of length at most $p$ has two occurrences at positions distinct by $p$.
    Hence, some marker has two such occurrences in $U$.
    Since each marker occurs only in a specific position on the border of two consecutive blocks, the difference between any two positions of a marker is a multiple of $k=|v|$.
    Therefore we conclude that $p=ck$ for some integer $c$.

    The final argument is the same as in Case 1 of Theorem~\ref{t:exponents}: $u_0$ is a suffix of $u_c$, $u_r$ is a prefix of $u_{r-c}$, and the maximality of the length of $U$ implies $u_0=u_c$, $u_r=u_{r-c}$. 
    Then $u=f_v^{-1}(U)$ is a factor of $w$ with the exponent $\exp(U)$.
\end{proof}
\begin{lemma} \label{l:notcubefree}
    Let $v$ be a factor of $\mathbf t$ having, for some $a\in\{0,1\}$, the word $a\,a$ or $a\,\overli{a}\,a\,\overli{a}$ as a prefix or a suffix. 
    Then the morphism $f_v$ is not cubefree.
\end{lemma}
\begin{proof}
    Either $v$ or $\overli{v}$ ends with $a$.
    If $v$ begins with $a\,a$, then one of the words $vv=f_v(00)$, $\overli{v}v=f_v(10)$ contains $a\,a\,a$.
    Hence $f_v$ is not cubefree.
    A symmetric argument works for the case where $v$ ends with $a\,a$.

    Now assume that the first two letters of $v$ are different, as well as its last two letters.
    Then one of $v,\overli{v}$ has the suffix $a\,\overli{a}$.
    If $v$ has the prefix $a\,\overli{a}\,a\,\overli{a}$, then one of the words $vv$, $\overli{v}v$ contains $(a\,\overli{a})^3$, and we see again that $f_v$ is not cubefree.
    A symmetric argument works for $v=\cdots a\,\overli{a}\,a\,\overli{a}$. Thus, $f_v$ is not cubefree in all cases listed in the lemma.
\end{proof}

\begin{proof}[Proof of Theorem~\ref{t:TMfact3free}]
    We first prove that the clipped \TM morphism $f_v$ is not cubefree if $k=|v|=3\cdot 2^n$.
    The proof is by induction on $n$, with the base case $n\in\{0,1\}$.
    If $n=0$, then either $f_v$ is not cubefree by Lemma~\ref{l:notcubefree} or
    $v=a\,\overli{a}\,a$ for some $a\in\{0,1\}$. 
    In the latter case, $f_v$ is not cubefree because $v\overli{v}=f_v(01)=(a\,\overli{a})^3$.
    If $n=1$, we have two cases for which Lemma~\ref{l:notcubefree} is not applicable: either $v=a\,\overli{a}\,\overli{a}\,a\,a\,\overli{a}$ or $v=a\,\overli{a}\,a\,a\,\overli{a}\,a$. 
    In the former case, $v\, \overli{v}$ is a cube, while in the latter case $v\, v$ is a fourth power.
    Thus, the base case holds.

    Now consider the induction step.
    If $v$ satisfies the conditions of Lemma~\ref{l:notcubefree}, there is nothing to prove, so suppose it does not. 
    Then there are two options for the prefix of $v$: either $v=a\,\overli{a}\,a\,a\,\overli{a}\cdots$ or $v=a\,\overli{a}\,\overli{a}\cdots$, for some $a\in\{0,1\}$.
    Consider the first option.
    If a position in $\mathbf{t}$ is divisible by 4, then it is preceded by a 2-block and is a first position of a 2-block. 
    Among the first four positions of $v$, only the second position can satisfy this property.
    Hence $v$ consists of the first letter, followed by a sequence of 2-blocks, and ends with a 3-letter prefix of a 2-block (recall that $n\ge 2$, which means that $|v|$ is divisible by 4).
    Then the last two letters of $v$ are equal.
    This contradicts our assumption that $v$ does not satisfy the conditions of Lemma~\ref{l:notcubefree}.
    Therefore, this option is not legal.
    
    Then $v=a\,\overli{a}\,\overli{a}\cdots$.
    Since $v$ is a factor of $\mathbf t$ of even length, this prefix of $v$ guarantees that $v=\mu(u)$ for some \TM factor $u$ of length $3\cdot 2^{n-1}$.
    By the inductive hypothesis, $f_u$ is not cubefree.
    Hence there exists a cubefree word $w$ such that $f_u(w)$ contains a cube. 
    Then $f_v(w)=\mu(f_u(w))$ also contains a cube, implying that $f_v$ is not cubefree.
    The step case is proved.

    In the remaining part of the proof we exhibit, for each $k=m\cdot 2^n$ such that $m$ is an odd integer distinct from 3, a cubefree morphism $f_v$ with $|v|=k$.
    For $m= 1$ we just refer to Theorem~\ref{t:muexponents}; 
    below we assume $m\ge 5$.
    (However, statement (a) stays true for $m=1$.)
    Consider $n=0$ first.
    
    Let $q$ be an integer. 
    At position $8q$, the \TM word contains a 3-block, i.e., the factor of the form $a\,\overli{a}\,\overli{a}\,a\,\overli{a}\,a\,a\,\overli{a}$.
    In particular, it contains the factor $\overli{a}\,\overli{a}\,a\,\overli{a}$, ending at position $8q+4$.
    As $\mathbf{t}[0..3]=\mathbf{t}[6..9]=0110$ and $\mathbf{t}[4..7]=1001$, the factors $v_1=\mathbf{t}[0..8q{+}4]$, $v_2=\mathbf{t}[4..8q{+}4]$, and $v_3=\mathbf{t}[6..8q{+}4]$ all satisfy the conditions of Lemma~\ref{l:cubefree}.
    Therefore, the morphisms $f_{v_1}=f_{0,8q+5}$, $f_{v_2}=f_{4,8q+1}$, and $f_{v_3}=f_{6,8(q-1)+7}$ are cubefree, and their lengths fall into the cases (c), (a), and (d), respectively, of the theorem. The case (b) needs a separate analysis.

    One of the basic properties of $\mathbf t$ (see \cite{Thue:1912}) says that $\mathbf{t}[q-1]=\mathbf{t}[q]$ if and only if $\mathbf{t}[2q-1]\ne \mathbf{t}[2q]$.
    Given that $\mu^3(\mathbf{t}[q{-}1..q])=\mathbf{t}[8q{-}8..8q{+}7]$ and $\mu(\mathbf{t}[8q{-}8..8q{+}7])=\mathbf{t}[16q{-}16..16q{+}15]$, the factor $\mathbf{t}[8q{-}8..16q{+}7]$ must have, for $a\in\{0,1\}$, one of the following two forms:
    \begin{align}
        \mathbf{t}[8q{-}8..16q{+}7] &= a\,\overli{a}\,\overli{a}\,a\,\overli{a}\,a\,\pmb{a\,\overli{a}\ \overli{a}\,a}\,a\,\overli{a}\,a\,\overli{a}\,\overli{a}\,a\cdots 
        \overli{a}\,a\,a\,\overli{a}\,a\,\pmb{\overli{a}\,\overli{a}\,a\ \overli{a}}\,a\,a\,\overli{a}\,a\,\overli{a}\,\overli{a}\,a \label{e:1st} 
        \\
        \mathbf{t}[8q{-}8..16q{+}7] &= a\,\overli{a}\,\overli{a}\,a\,\overli{a}\,a\,a\,\pmb{\overli{a}\ a\,\overli{a}\,\overli{a}}\,a\,\overli{a}\,a\,a\,\overli{a}\cdots 
        \overli{a}\,a\,a\,\overli{a}\,a\,\overli{a}\,\pmb{\overli{a}\,a\ a\,\overli{a}}\,\overli{a}\,a\,\overli{a}\,a\,a\,\,\overli{a}. \label{e:2nd} 
    \end{align}
    The word $v_4=\mathbf{t}[8q{-}2..16q]$ in case of form~\eqref{e:1st} and the word $v_5=\mathbf{t}[8q{-}1..16q{+}1]$ in case of form~\eqref{e:2nd} satisfy the conditions of Lemma~\ref{l:cubefree}.
    Therefore, for each integer $q$ either $f_{v_4}=f_{8q-2,8q+3}$ or $f_{v_5}=f_{8q-1,8q+3}$ is cubefree.
    Hence we are done with the case (b), and thus with all cases, provided that $n=0$.

    Finally, if $f_v=f_{\ell,m}$ is cubefree, then $f_{\mu^n(v)}=f_{\ell\cdot 2^n,m\cdot 2^n}$ is cubefree because $\mu$ is cubefree.
    This observation finishes the proof for arbitrary $n$.    
\end{proof}

\begin{proof}[Proof of Theorem~\ref{t:comp3free}]
    It suffices to consider the case $k=3\cdot 2^n$, as Theorem~\ref{t:TMfact3free} provides complementary morphisms of all other lengths.
    The complementary morphism $g: 0\mapsto 011001001101, 1\mapsto 100110110010$ of length 12 satisfies Kobayashi's condition (Theorem~\ref{t:kobayashi}) for $\alpha=3$, and hence is cubefree.
    Then $g\mu^{n-2}$ is a cubefree complementary morphism of length $3\cdot 2^n$ for each $n\ge 2$.
    An easy search shows that all complementary morphisms of lengths 3 and 6 are not cubefree.
    The theorem is proved.
\end{proof}

\begin{remark} \label{r:3uniform}
Since the morphism $\rho: 0\mapsto 001, 1\mapsto 011$ is cubefree (see, e. g., \cite{Currie&Rampersad:2009}) and hence $\rho\mu$ is cubefree as well, Theorem~\ref{t:comp3free} provides an alternative proof of Theorem~\ref{t:3free_lengths}.
Instead of $\rho$, the cubefree morphism $\psi: 0\mapsto 010, 1\mapsto 011$ (see, e. g., \cite{Petrova&Shur:2012}) can be taken.
\end{remark}

\subsection{Proof of Theorem~\ref{t:below3}}

We start with statement (a). Observe that if $h$ is a $k$-uniform $\alpha$-free binary morphism, where $k>1$, then $\alpha\ge 2^+$.
Then the condition $\alpha<3$ means that $\alpha$ is not an integer.
Therefore, $h$ must satisfy the conditions of Observation~\ref{o:necessary}.
A short search shows that for each uniform binary morphism $h$ that satisfies the conditions of Observation~\ref{o:necessary} and has length 3 or 6, there is a word $x$ of length 2 such that $h(x)$ contains a cube.
Hence $h$ is not $\alpha$-free for any $\alpha<3$.
Statement (a) is proved.


For statement (b), we make use of a simple refinement on Kobayashi's condition.
\begin{lemma} \label{l:minkobayashi}
    If a morphism $h$ satisfies all conditions of Theorem~\ref{t:kobayashi} for some $\alpha$, then the minimum such $\alpha$ has the form $\beta^+$ for a rational $\beta$.
\end{lemma}
\begin{proof}
    Assume that $h$ verifies the conditions of Theorem~\ref{t:kobayashi} for some $\alpha\in\mathbb{R}$ (without a plus!) and show that $\alpha$ is not minimal.
    Condition (c), which is the only one dependent on $\alpha$, says that $\ce(h(x))<\alpha$ whenever $|x|\le \lceil\alpha\rceil+1$ and $\ce(x)<\alpha$. 
    Taking the maximum critical exponent over a finite set of these finite words $h(x)$, one gets a rational $\beta=\max\{\ce(h(x))\mid |x|\le \lceil\alpha\rceil+1, \ce(x)<\alpha\}$.
    Then $\beta<\alpha$ and $h$ verifies the conditions of Theorem~\ref{t:kobayashi} for $\beta^+$.
    As $\alpha\in\mathbb{R}$, one gets $\beta^+<\alpha$; the lemma now follows.
\end{proof}
By Lemma~\ref{l:minkobayashi}, the morphism 
$g: 0\mapsto 011001001101, 1\mapsto 100110110010$ from the proof of Theorem~\ref{t:comp3free} is $\alpha$-free for some $\alpha<3$.
Then all morphisms of the form $g\mu^n$ are $\alpha$-free as well, so we get statement (b) for all lengths $k$ of the form $k=3\cdot 2^n$.

However, some clipped \TM morphisms violate condition (b3) of Theorem~\ref{t:kobayashi}.
For instance, the word $01101$ is both a prefix of $h(0)$ and a suffix of $h(1)$ for the morphism $h: 0\mapsto 011010010, 1\mapsto 100101101$. 
Hence, instead of checking whether all morphisms listed in Theorem~\ref{t:TMfact3free} satisfy this condition, we use a different argument.
As all odd-length morphisms from Theorem~\ref{t:TMfact3free} satisfy the conditions of Lemma~\ref{l:cubefree}, we strengthen the conclusion of Lemma~\ref{l:cubefree}.
Namely, we show that the morphism $f_v$ is $(3-1/k)^+$-free. 

Let us follow the proof of Lemma~\ref{l:cubefree}.
For an arbitrary fixed period $p$, we study the longest factor $U$ of $W=f_v(u)$ satisfying $\per(U)=p$, and analyse the case where $\exp(U)\ge 3$.
If $p\le k$, then the conditions $\exp(U)\ge 3$ and $\exp(U)> 3-1/k$ are equivalent.
In this case, the proof of the lemma guarantees the existence of a factor $u$ of $v$ such that $\exp(u)=\exp(U)$.
Now let $p>k$.
Note that the conditions on the prefixes and suffixes of $v$ imply $k\ge 5$.
We have $\lceil (1-1/k)\cdot p\rceil\ge \lceil (1-1/5)\cdot 6\rceil = 5$. 
Then the condition $\exp(U)> 3-1/k$ implies $|U|\ge \lceil(3-1/k)\cdot p\rceil\ge 2p+5$.
Hence $U$ contains a marker; let $x=U[i..i{+}4]$ be a marker.
If $i\ge p$, then $x=U[i{-}p..i{-}p{+}4]$; otherwise, $x=U[i{+}p..i{+}p{+}4]$.
In both cases, the marker $x$ occurs in $U$ twice at positions distinct by $p$.
Having proved this, we follow the rest of the proof of the lemma, finding the factor $u$ of $v$ such that $\exp(u)=\exp(U)$.
The existence of such a factor $u$ whenever $\exp(U)> 3-1/k$ ensures that $f_v$ is $(3-1/k)^+$-free.

Thus we have proved that for each odd integer $k$ the complementary morphism of length $k$ listed in Theorem~\ref{t:TMfact3free} is $(3-1/k)^+$-free.
Since every even-length morphism from this theorem is the composition of an odd-length morphism and an appropriate power of $\mu$, we obtain, for each length $k$ distinct from $3\cdot 2^n$, a complementary morphism that is $\alpha$-free for some $\alpha<3$.
Statement (b) of Theorem~\ref{t:below3} is thus proved.

\section{\texttt{Walnut}-Assisted Proofs} \label{s:walnut}

\subsection{Automata theory and \texttt{Walnut}} \label{s:walnutbasics}

In this section we review the automata theory we
need to discuss our proofs using \texttt{Walnut},
and give examples of how \texttt{Walnut} can be used to prove results.

A DFAO (deterministic finite automaton with output) is a slight generalization of the more familiar DFA (deterministic finite automaton).  In it, each state has an associated output.  The output associated with an input word $x$ is the output of the last state reached when processing $x$.
We say that a sequence
${\bf a} = (a_n)_{n \geq 0}$ is {\it $k$-automatic\/} if there exists a DFAO that, on input $n$ expressed in base $k$, outputs $a_n$. 
See \cite{Allouche&Shallit:2003} for more information.

A basic result that we will rely on throughout is the following \cite[Thm.~6.1, Cor.~6.2]{Bruyere&Hansel&Michaux&Villemaire:1994}:
\begin{theorem}
There is an algorithm that, given a
$k$-automatic sequence ${\bf a} = (a_n)_{n \geq 0}$ and a first-order logical formula $\varphi$ in the structure
$\langle \Enn, +, n \rightarrow a_n \rangle$, computes a DFA $M_\varphi$ that takes the unbound variables of $\varphi$ as inputs, and accepts precisely those values of the variables (represented in base $k$) that make $\varphi$ true.   If $\varphi$ has no free variables, the algorithm returns {\tt TRUE} or {\tt FALSE}, and accordingly constitutes a rigorous proof or disproof of the formula. 
\end{theorem}

Furthermore, this algorithm has been implemented in free software called \texttt{Walnut} \cite{Mousavi:2016,Shallit:2022}.
As a warm-up, we provide \texttt{Walnut} proofs
of Lemmas~\ref{l:blocks}, \ref{l:intersection}, and \ref{l:periods}.

\begin{proof}[Alternative proof of Lemma~\ref{l:blocks}]
Since there is no way to express $2^n$ directly in the logical theory that \texttt{Walnut} is based on, instead we let $x$ play the role of $2^n$, and simply assert that $x$ is some (unspecified) power of $2$.

Without loss of generality we can assume
that $a = 0$.  Supposing that $x = 2^n$, we have $\mu^n(a \overli{a}) =
{\bf t}[0..2x-1]$, and 
$\mu^n(aa) = {\bf t}[5x..7x-1]$.  Then the
lemma's assertions can be translated into
\texttt{Walnut} as follows:
\begin{verbatim}
reg power2 msd_2 "0*10*":
def same "At (t<n) => T[i+t]=T[j+t]":
# asserts that factors t[i..i+n-1] and t[j..j+n-1] are identical
def complementary "At (t<n) => T[i+t]!=T[j+t]":
# asserts that factors t[i..i+n-1] and t[j..j+n-1] are complementary

eval lemma11_a_1 "Ai,x ($power2(x) & $same(0,i,x) & i<=x) => (i=0|i=x)":
eval lemma11_a_2 "Ai,x ($power2(x) & $complementary(0,i,x) & i<=x) => 
   (i=0|i=x)":

eval lemma11_b_1 "Ai,x ($power2(x) & $same(5*x,5*x+i,x) & i<=x) => 
   (i=0|i=x)":
eval lemma11_b_2 "Ai,x ($power2(x) & $complementary(5*x,5*x+i,x) & i<=x) =>
   (i=0|2*i=x)":
\end{verbatim}
and \texttt{Walnut} returns {\tt TRUE} for all of them.
\end{proof}
Hopefully, \texttt{Walnut}'s syntax is fairly clear, but here are some brief explanations.  The {\tt reg} command translates a regular expression to an automaton; the {\tt def} command defines an automaton with inputs given by the free variables of the expression; the {\tt eval} command evaluates the truth or falsity of an assertion with no free variables.
The letters {\tt A} and {\tt E} represent the universal and existential quantifiers, respectively.   The letter {\tt T} refers to the infinite \TM word $\bf t$.  
The symbol {\tt \&} denotes logical ``and'',
{\tt |} denotes logical ``or'', and
{\tt =>} denotes implication. 

\begin{proof}[Alternative proof of Lemma~\ref{l:intersection}.]
We only need to handle the case of a word $u$ of length at most
$2k$.   Let $n= |u|$.  In this case we are comparing
${\bf t}[0..2k-n-1]$ to ${\bf t}[n-k..k-1]$.

\begin{verbatim}
def overlap_eq "n>k & n<=2*k & $same(0,n-k,2*k-n)":
def overlap_comp "n>k & n<=2*k & $complementary(0,n-k,2*k-n)":
eval lemma12_1a "Ak,n $overlap_eq(k,n) => 2*n>=3*k":
eval lemma12_1b "Ak,n  ($overlap_eq(k,n) & 2*n=3*k) => $power2(k)":
eval lemma12_2 "Ak,n $overlap_comp(k,n) => 2*n>=3*k":
eval lemma12_2b "Ak,n  ($overlap_comp(k,n) & 2*n=3*k) => $power2(k)":
\end{verbatim}
and \texttt{Walnut} returns {\tt TRUE} for all of them.
\end{proof}

\begin{proof}[Alternative proof of Lemma~\ref{l:periods}]
Again,
we let $x$ play the role of $2^n$.
First we create a synchronized automaton that
accepts $p = m \cdot x =  m \cdot 2^n$ in parallel with $x=2^n$.  
\begin{verbatim}
reg twopower msd_2 msd_2 "[0,0]([1,0]|[0,0])*[1,1][0,0]*":
\end{verbatim}

Next we create a synchronized automaton for $(i,n,p)$, asserting that
$p$ is a period of ${\bf t}[i..i+n-1]$.

\begin{verbatim}
def tmper "p>0 & p<=n & Aj (j>=i & j+p<i+n) => T[j]=T[j+p]":
\end{verbatim}

Now we rewrite the assertions using the following
\texttt{Walnut} code.
\begin{verbatim}
eval casea "Ai,n,p,x ($tmper(i,n,p) & $twopower(p,x)) => n<=p+3*x":

eval caseb_i "Ak,p,x ($tmper(0,k,p) & p<k & $twopower(p,x)) => k!=x":
eval caseb_ii "Ak,p,x ($tmper(0,k,p) & p<k & $twopower(p,x) & p=3*x) => 
   k<=5*x":
eval caseb_iii "Ak,p,x ($tmper(0,k,p) & p<k & $twopower(p,x) & p=5*x) => 
   k<=6*x":
\end{verbatim}
and \texttt{Walnut} returns {\tt TRUE} for all of them.
\end{proof}

The reader is invited to contrast the more traditional proofs of Lemmas~\ref{l:blocks}, \ref{l:intersection}, and \ref{l:periods} given in Section~\ref{s:trunc} with the ones we constructed here using \texttt{Walnut}.

\subsection{Alternative proof of the case $\alpha=3^+$ of Theorem~\ref{t:3plusfree}} \label{s:walnut3plus}

We now use \texttt{Walnut} to provide an alternative proof of the crucial case of Theorem~\ref{t:3plusfree}.
The proof is based on Kobayashi's conditions (Theorem~\ref{t:kobayashi}).
For every morphism $h_k$ we translate each of the conditions of Theorem~\ref{t:kobayashi} for the exponent $\alpha = 3^+$ into \texttt{Walnut}, and let \texttt{Walnut} verify the claims.


\begin{proof}[Alternative proof of Theorem~\ref{t:3plusfree} (case $\alpha=3^+$)]
Conditions (b1) and (b2) follow trivially from the definition of $h_k$. 
For the other conditions, we use \texttt{Walnut} as described above.  

We start with condition (c).  By the criterion, we must apply
$h_k$ to all $3^+$-free binary words $w$ of length $4$ and verify
that $h_k (w)$ is $3^+$-free in all cases.  Of the $16$ possible
binary words, only $0000$ and $1111$ are not $3^+$-free, so at first
glance it would seem that we have to check the condition (c) for $14$
binary words.

However, to save effort and computation time, there are some symmetries we can exploit to avoid having to test all $14$ words.  Given a morphism
$h$, we can define its reverse morphism $h^R (a) = h(a)^R$ by
reversing the image of each letter.   Then it is easy to see that
$h(w) = (h^R (w))^R$.  On the other hand, for the particular case of
$h_k$, by the definition we have $h_k (w) =
\overline{ h_k( \overline{w})} $.   Since $\ce(w) = \ce(w^R)$ and
$\ce(w) = \ce(\overline{w})$, this means it suffices to test the $3^+$-free
property for five words:  $0001$, $0010$, $0011$, $0101$, and $0110$.

We now create first-order logical formulas to test the condition.  We illustrate
the idea for the specific case of $0001$.   The
idea is to create a formula that asserts that $(h_k (0001))[n] = 1$.

In order to do that, we must be able to compute $n \bmod k$ and
$\lfloor n/k \rfloor$ with 
a first-order formula, for variables $n$ and $k$.  In full generality, this is provably impossible
to state in the particular first-order logical theory that
\texttt{Walnut} uses, so we have to use some subterfuge:   since we only
care about the range
$0 \leq n < 4k$, it actually suffices to be able to compute $n \bmod k$
just in this range, and not for all $n$.
We use the following formula:
\begin{multline*}
    \ndmodk(k,n,r,z) := 
(0\leq n<k \andd z=n \andd r=0) \orr 
(k\leq n<2k \andd z=n-k \andd r=1) \orr \\
(2k\leq n<3k \andd z=n-2k \andd r=2)\orr
(3k\leq n<4k \andd z=n-3k \andd r=3),
\end{multline*}
which evaluates to
{\tt TRUE}
if and only if $0 \leq n < 4k$ and
$r = \lfloor n/k \rfloor$ and $z = n \bmod k$.

This can be translated to {\tt Walnut\/} as follows:
\begin{verbatim}
def ndmodk "(n<k & z=n & r=0) | (n>=k & n<2*k & z+k=n & r=1) |
   (n>=2*k & n<3*k & z+2*k=n & r=2) | (n>=3*k & n<4*k & z+3*k=n & r=3)":
\end{verbatim}
The resulting automaton has $40$ states.

Next, we make a logical formula that evaluates
to {\tt TRUE} if $h_k(0001)[n] = 1$:
\begin{align*}
\hone(k,n) &:= (\exists r,z \ r\leq 2 \andd \ndmodk(k,n,r,z) \andd {\bf t}[z]=1) \orr \\
&\quad\quad (\exists r,z \ r=3 \andd \ndmodk(k,n,r,z) \andd {\bf t}[z]=0).
\end{align*}
This can be translated into \texttt{Walnut} as follows:
\begin{verbatim}
def h0001 "(Er,z r<=2 & $ndmodk(k,n,r,z) & T[z]=@1) | 
   (Er,z r=3 & $ndmodk(k,n,r,z) & T[z]=@0)":
\end{verbatim}
The resulting automaton has 71 states.

For the other four strings, we use analogous formulas, as follows:
\begin{verbatim}
def h0010 "(Er,z (r<=1|r=3) & $ndmodk(k,n,r,z) & T[z]=@1) | 
   (Er,z r=2 & $ndmodk(k,n,r,z) & T[z]=@0)":
def h0011 "(Er,z r<=1 & $ndmodk(k,n,r,z) & T[z]=@1) | 
   (Er,z (r=2|r=3) & $ndmodk(k,n,r,z) & T[z]=@0)":
def h0101 "(Er,z (r=0|r=2) & $ndmodk(k,n,r,z) & T[z]=@1) | 
   (Er,z (r=1|r=3) & $ndmodk(k,n,r,z) & T[z]=@0)":
def h0110 "(Er,z (r=0|r=3) & $ndmodk(k,n,r,z) & T[z]=@1) | 
   (Er,z (r=1|r=2) & $ndmodk(k,n,r,z) & T[z]=@0)":
\end{verbatim}
Each one of these commands produces a $71$-state automaton.

Now we create a formula that asserts that $(h_k(0001))[i..i+2n]$ has
period $n$:
$$ \testone(i,k,n) := \forall u, v\ (i \leq u\leq i+2n \andd v=u+n) \implies 
(\hone(k,u) \iff \hone(k,v)) .$$
In \texttt{Walnut} this is as follows:
\begin{verbatim}
def test0001 "Au,v (u>=i & u<=i+2*n & v=u+n) => 
   ($h0001(k,u) <=> $h0001(k,v))"::
# 27 states
\end{verbatim}

Finally, we create a formula that asserts that there do not exist $k, i, n$
with $k,n \geq 1$ such that
$h_k(0001)$ contains a factor that is a
$3^+$-power:
$$ \checkone := \neg\exists k,i,n \ (n \geq 1) \andd (k \geq 1) \andd (i+3n<4k) \andd \testone(i,k,n).$$
In \texttt{Walnut} this is
\begin{verbatim}
eval check0001 "~Ek,i,n n>=1 & k>=1 & (i+3*n<4*k) & $test0001(i,k,n)":
\end{verbatim}
and \texttt{Walnut} returns
{\tt TRUE}.   The analogous expressions for the strings $0010, 0011, 0101, 0110$ have almost exactly the same behavior, and all return {\tt TRUE}.

Finally, conditions (a) and (b3) of Theorem~\ref{t:kobayashi} can be verified in a very similar manner by determining, for which
pairs $(i,k)$ with
$k \geq 1$ and $k/2 \leq i < k$ we have
\begin{itemize}
    \item $\mathbf{t}[0..i{-}1]=\mathbf{t}[k{-}i..k{-}1]$ for condition (a),
    \item $\mathbf{t}[0..i{-}1]=\overline{\mathbf{t}[k{-}i..k{-}1]}$ for condition (b3).
\end{itemize}
We check these with the following \texttt{Walnut} code:
\begin{verbatim}
def tpsa "k>0 & 2*i>=k & i<k & Aj (j<i) => T[j]=T[k+j-i]":
def tpsb "k>0 & 2*i>=k & i<k & Aj (j<i) => T[j]!=T[k+j-i]":
\end{verbatim}
In response, \texttt{Walnut} produces two automata accepting the possible pairs $(i,k)$ for each of two cases.
The automaton for {\tt tpsa} accepts nothing (as $\mathbf{t}$ is overlap-free and contains no squares among its prefixes)  and thus verifies condition (a).
The automaton for {\tt tpsb} only accepts the pairs $(i,k) = (2^n, 2^{n+1})$ for $n \geq 0$.   In this case it is easy to see that condition (b3) is satisfied, with $u=\mu^n(0)$ and $v=\mu^n(1)$.
The reference to Theorem~\ref{t:kobayashi} completes the proof of Theorem~\ref{t:3plusfree} for $\alpha=3^+$.
\end{proof}

\begin{remark}
By far, the most resource-intensive part of our \texttt{Walnut} computation for Theorem~\ref{t:3plusfree} was the calculation of the five {\tt test}
automata.  Each one required about 350G of RAM and used about 75,000 seconds of CPU time. Computing
the five automata involved, as an intermediate step, minimizing a DFA of 
162,901,489 states.  
\end{remark}

\subsection{New cubefree morphisms and another proof of Theorem~\ref{t:3free_lengths}} \label{s:walnut3free}

Now we turn to the case of cubefree morphisms.
We start with adding to the result of Theorem~\ref{t:exponents}. Namely, we determine for which $k$ the truncated \TM morphisms $h_k$ are cubefree.  
Clearly, this is not always the case: for example, $0010$ is cubefree, but
$h_3(0010) = 011011100011$ is not.

Similar to Section~\ref{s:walnut3plus}, we encode Kobayashi's conditions with formulas and verify them by \texttt{Walnut}. 
Note that our result is stated for $h_{k+1}$ (it is possible to state it for $h_k$, but the automaton will have more states and less clear structure).

\begin{theorem} \label{t:automaton}
    The morphism $h_{k+1}$ is cubefree if and only if the base-$2$ representation of $k$ is accepted by the automaton $\cA$ depicted in Figure~\ref{f:tm3}.
    \begin{figure}[H]
    \begin{center}
    \includegraphics[scale=0.9, trim = 30 39.5 30 41, clip]{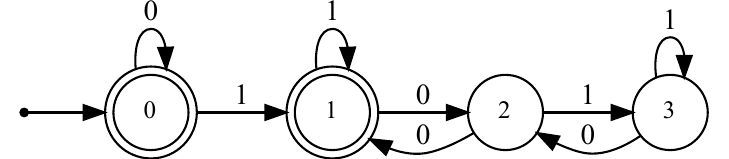}
    \end{center}
    \caption{Automaton $\cA$ accepting those $k$ for which $h_{k+1}$ is a cubefree morphism.}
    \label{f:tm3}
    \end{figure}
\end{theorem}

\begin{proof} 
Conditions (a) and (b1)--(b3) of Theorem~\ref{t:kobayashi} are independent of $\alpha$ and were verified for all morphisms $h_k$ in Section~\ref{s:walnut3plus}.
Thus it suffices to check condition (c) for $\alpha=3$ and the words $0010$, $0011$, $0101$, and $0110$.
We use the following 
\texttt{Walnut} code to test the word $0010$:
\begin{verbatim}
def h0010b "(Er,z (r<=1|r=3) & $ndmodk(k,n,r,z) & T[z]=@1) | 
   (Er,z r=2 & $ndmodk (k,n,r,z) & T[z]=@0)"::
def test0010b "Au,v (u>=i & u<i+2*n & v=u+n) => 
   ($h0010b(k,u) <=> $h0010b(k,v))"::
def check0010b "~Ei,n n>=1 & (i+3*n<4*k) & $test0010b(i,k,n)"::
\end{verbatim}
The code to test the other three words is analogous to the above one.
Then we construct a single automaton that combines all these tests:
\begin{verbatim}
def hcubefree "$check0010b(k+1) & $check0011b(k+1) & $check0101b(k+1) &
   $check0011b(k+1)":
\end{verbatim}
Evaluating \texttt{hcubefree}, \texttt{Walnut} returns exactly the automaton $\cA$.
\end{proof}

\begin{remark}
This was another significant computation.
Proving these results involved minimizing automata of 162,901,489 states.  The typical CPU time required was about 76,000  seconds for each of the four words.
\end{remark}
\begin{remark}
    If we consider a uniformly random walk in the automaton $\cA$, then the limit probability equals $0$ for state $0$ and $1/3$ for each of the other three states.
    Therefore, the set of numbers accepted by $\cA$ has density $1/3$ in $\mathbb{N}$.
    In other words, approximately one of three morphisms $h_k$ is cubefree.
\end{remark}
\begin{remark}
    The automaton in Figure~\ref{f:tm3} accepts, in particular, all words having the suffix 100.
    Hence, if $k\equiv 4\pmod 8$, then $h_{k+1}$ is cubefree, which is exactly statement (c) of Theorem~\ref{t:TMfact3free} (recall that $f_{0,k}=h_k$).
    Thus, Theorem~\ref{t:automaton} implies this statement.
    However, our attempt to verify Kobayashi's conditions for the morphisms from statement (b) failed due to resource limitations.
    Thus, we were unable to get an ``automatic'' proof of Theorem~\ref{t:TMfact3free}.    
\end{remark}

Nevertheless, we can use \texttt{Walnut} to get another proof of Theorem~\ref{t:3free_lengths}.
For this proof, we use morphisms based on the truncations of two different underlying automatic sequences.  
Together, these two families of morphisms handle all $k \equiv 1,5\pmod 6$.

The two sequences we use are
\begin{align*}
{\bf x} &= 0 \, \overline{\bf t} = 0 10010110 \cdots\\
&\text{and} \\
{\bf y} &= 0\, \mu(1)\, \mu^2(0)\, \mu^3(1)\, \mu^4(0) \cdots = 01001101001011001101001100101101001011 \cdots 
\end{align*}
This latter sequence $\bf y$ is
sequence \seqnum{A059448} in the OEIS. 
It is the unique
infinite word satisfying the equation
${\bf y} = 010 \mu^2({\bf y})$.
The sequence $0\mathbf y$ appeared in literature \cite{Shallit:2011,Du&Shallit&Shur:2015,Shallit&Shur:2019} under the name ``twisted \TM sequence''.
This sequence counts the parity of the number of nonleading $0$'s in the binary representation of $n$. 

The automata generating these two automatic sequences are depicted in Figures~\ref{f:x} and
\ref{f:y}, respectively.
\begin{figure}[!htb]
\centering
\includegraphics[scale=0.9, trim = 30 40 30 40, clip]{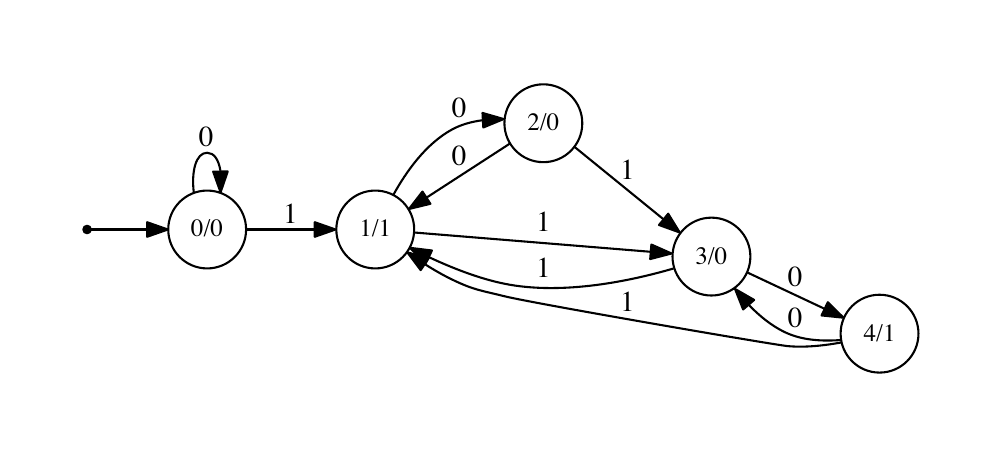}
\caption{Automaton for the sequence {\bf x}.}
\label{f:x}
\end{figure}
\begin{figure}[!htb]
\centering
\includegraphics[scale=0.9, trim = 30 40 30 40, clip]{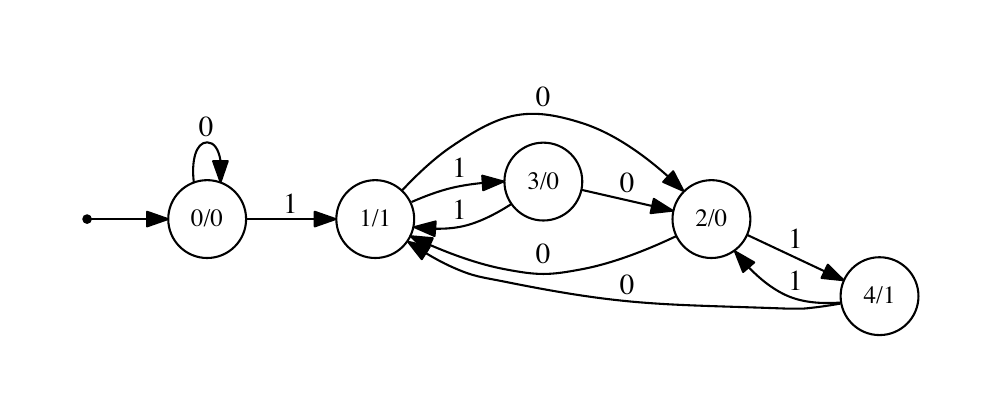}
\caption{Automaton for the sequence {\bf y}.}
\label{f:y}
\end{figure}

We now define morphisms $f_k$ (based on $\bf x$) and $g_k$ (based on $\bf y$), in analogy with $h_k$ defined previously:
$f_k$ maps $0$ to the first $k$ bits of $\bf x$ and $1$ to the complement of $f_k(0)$, while
$g_k$ maps $0$ to the first $k$ bits of
$\bf y$ and $1$ to the complement of
$g_k (0)$.

We then carry out exactly the same kind of proof as we did for Theorem~\ref{t:automaton}.
Here is the \texttt{Walnut} code for testing the cubefreeness of the morphisms $f_k$ and $g_k$ applied to $0010$; the analogous code for the other three words should be clear.
\begin{verbatim}
def f0010 "(Er,z (r<=1|r=3) & $ndmodk(k,n,r,z) & X[z]=@1) |
   (Er,z r=2 & $ndmodk(k,n,r,z) & X[z]=@0)":
def testf0010 "Au,v (u>=i & u<i+2*n & v=u+n) => 
   ($f0010(k,u) <=> $f0010(k,v))"::
def checkf0010 "~Ei,n n>=1 & (i+3*n<4*k) & $testf0010(i,k,n)"::
def g0010 "(Er,z (r<=1|r=3) & $ndmodk(k,n,r,z) & Y[z]=@1) |
   (Er,z r=2 & $ndmodk(k,n,r,z) & Y[z]=@0)":
def testg0010 "Ax (x<2*n) =>($g0010(k,i+x) <=> $g0010(k,i+n+x))"::
def checkg0010 "~Ei,n n>=1 & (i+3*n<4*k) & $testg0010(i,k,n)"::
\end{verbatim}
The careful reader will notice a small difference in how we defined {\tt testf0010} and {\tt testg0010}; this was
used to improve efficiency in the computation.

Once we have the analogous code for the other three words, we can prove Theorem~\ref{t:3free_lengths}.

\begin{proof}
If $k \equiv \modd{1,5} {6}$, then we claim either $f_k$ or $g_k$ satisfies condition (c) of Theorem~\ref{t:kobayashi}.  More precisely,
$k \equiv \modd{1,5} {12}$ is handled by
$f_k$ and $k \equiv \modd{7,11} {12}$ is handled
by $g_k$.  To see this, we use the following \texttt{Walnut} code:
\begin{verbatim}
def mod12 "Ej k=12*j+r":
def fcubefree "$checkf0010(k) & $checkf0011(k) & $checkf0101(k) &
   $checkf0110(k)":
def gcubefree "$checkg0010(k) & $checkg0011(k) & $checkg0101(k) &
   $checkg0110(k)":
eval thm6a "Ak ($mod12(k,1)|$mod12(k,5)) => $fcubefree(k)":
eval thm6b "Ak ($mod12(k,7)|$mod12(k,11)) => $gcubefree(k)":
\end{verbatim}
and \texttt{Walnut} returns {\tt TRUE} for both of
the last two commands.

We also need to check conditions (a) and (b3).  These can be checked as in the proof of Theorem~\ref{t:3plusfree},
as follows:
\begin{verbatim}
def x_pref_suff_00 "k>0 & 2*i>=k & i<k & Aj (j<i) => X[j]=X[k+j-i]": 
def x_pref_suff_01 "k>0 & 2*i>=k & i<k & Aj (j<i) => X[j]!=X[k+j-i]":
def y_pref_suff_00 "k>0 & 2*i>=k & i<k & Aj (j<i) => Y[j]=Y[k+j-i]": 
def y_pref_suff_01 "k>0 & 2*i>=k & i<k & Aj (j<i) => Y[j]!=Y[k+j-i]": 
\end{verbatim}
This \texttt{Walnut} code constructs automata accepting the pairs $(i,k)$ for which 
$k \geq 1$, $k/2 \leq i < k$, where
${\bf x}[0..i{-}1] = {\bf x}[k{-}i..k{-}1]$ (resp.,
${\bf x}[0..i{-}1] = \overline{{\bf x}[k{-}i..k{-}1]}$), and the same for $\bf y$.
After inspection of the resulting automata, we easily see that
\begin{itemize}
    \item {\tt x\_pref\_suff\_00} accepts only
    the pair $(3,6)$;
    \item {\tt x\_pref\_suff\_01} accepts only the
    pairs $(2^{2i}, 2^{2i+1})$ and
    $(2^{2i} + 1, 2^{2i+1} + 1)$ for 
    $i \geq 0$;
    \item {\tt y\_pref\_suff\_00} accepts nothing;
    \item {\tt y\_pref\_suff\_01} accepts only the
    pairs $(1,2)$ and $(2,3)$.
\end{itemize}
Checking each case shows that Kobayashi's conditions
(a) and (b3) never fail when $k \equiv \modd{1,5} {6}$; in fact, the only time they are not satisfied is for $\bf x$ and $(i,k) = (2^{2i} + 1, 2^{2i+1} + 1)$ for 
    $i \geq 0$.  But in this case $k$ is divisible by $3$.
    
We have therefore demonstrated the existence of a $k$-uniform cubefree morphism when $k \equiv \modd{1,5} {6}$.

Otherwise write $k = 2^a 3^b c$, where
$c \equiv \modd{1,5} {6}$. 
Then the desired morphism is the composition $\mu^a \circ \rho^b \circ f_c$ or $\mu^a\circ \rho^b \circ g_c$, according to whether $c \equiv \modd{1,5} {12}$ or $c \equiv \modd{7,11} {12}$.
Recall that $\rho: 0 \mapsto 001, 1 \mapsto
011$ is the morphism used by Currie and Rampersad \cite{Currie&Rampersad:2009}.
\end{proof}

\begin{remark}
    This was yet another significant computation.
    Proving the results for $f_k$ involved minimizing automata of 187,445,235 states.  The typical CPU time required was about 105,000  seconds for each of the four words.
    Proving the results for $g_k$ involved minimizing automata of 200,258,107 states.  The typical CPU time required was about 110,000
    seconds for each of the four words.
\end{remark}



\section{Discussion} \label{s:discussion}

1. In this paper, we study $\alpha$-free uniform binary morphisms.
Our Theorems~\ref{t:3plusfree}--\ref{t:below3}, together with earlier known Theorem~\ref{t:3free_lengths}, reveal a ``phase transition'' at $\alpha=3$.
Namely, $\alpha$-free \emph{uniform} binary morphisms of all lengths exist for every $\alpha\ge 3$, but for every $\alpha<3$ there are some ``forbidden'' lengths.
And if we consider only \emph{complementary} morphisms instead of arbitrary uniform binary morphisms, then we get almost the same picture: the only difference is that forbidden lengths exist already for $\alpha=3$.

Given the characterization of $\alpha$-free morphisms for $\alpha\le 7/3$ (Theorem~\ref{t:73free}), a natural open problem is to characterize the phase transitions in the remaining interval $(7/3)^+\le\alpha<3$. 
We conjecture that the set of admissible lengths of morphisms stays the same whenever $(8/3)^+\le\alpha<3$, but there is another phase transition at $8/3$.

\medskip\noindent
2. We note that there exist other properties separating cubefree and $3^+$-free words in the binary case.
For example, infinite cubefree words must contain squares of even period while $3^+$-free words can avoid them \cite{Badkobeh&Crochemore:2015}. 

In a recent paper, Meleshko et al.~\cite{Meleshko} studied infinite pseudoperiodic words.
We say an infinite word $\bf x$ is $( p_1, p_2, \ldots, p_t )$-pseudoperiodic if for all $i$ there exists $j$, $1 \leq j \leq t$, such that ${\bf x}[i] = {\bf x}[i+p_j]$.
Their Conjecture 34 asserts that for all pairs $1 \leq a < b$, $b \not=2a$, there is an infinite binary $(a,b)$-pseudoperiodic word
that is $3^+$-free.
If the conjecture holds, it indicates another phase transition at $\alpha=3$, because for some pairs $(a,b)$ the word with required properties cannot be cubefree.

Our Theorem~\ref{t:3plusfree} shows that, in order to prove this conjecture, it suffices to prove it in the case where $a$ and $b$ are relatively prime.
For if $d = \gcd(a,b) > 1$, then we could find an infinite binary word $\bf w$ avoiding $(a/d,b/d)$-pseudopowers that is $3^+$-free, and then apply $h_g$.
This was, in fact, the first piece of motivation for the study presented in this paper.

\section*{Acknowledgments}

We are very grateful to Gw\'ena\"el Richomme, who reminded us about Kobayashi's conditions
\cite[Thm. 5.3]{Kobayashi:1986}, heavily used in this study.

Many of the computations in this paper were done on the CrySP RIPPLE Facility at the University of Waterloo.  We express our deep thanks to Ian Goldberg for allowing us to run our computations on this machine.



\begin{thebibliography}{10}

\bibitem{Allouche&Shallit:2003}
Jean{-}Paul Allouche and Jeffrey~O. Shallit.
\newblock {\em Automatic Sequences: Theory, Applications, Generalizations}.
\newblock Cambridge University Press, 2003.

\bibitem{Badkobeh&Crochemore:2015}
Golnaz Badkobeh and Maxime Crochemore.
\newblock Infinite binary words containing repetitions of odd period.
\newblock {\em Inf. Process. Lett.}, 115(5):543--547, 2015.

\bibitem{Bean&Ehrenfeucht&McNulty:1979}
D.~A. Bean, A.~Ehrenfeucht, and G.~McNulty.
\newblock Avoidable patterns in strings of symbols.
\newblock {\em Pacific J. Math.}, 85:261--294, 1979.

\bibitem{Berstel&Seebold:1993}
Jean Berstel and Patrice S{\'{e}}{\'{e}}bold.
\newblock A characterization of overlap-free morphisms.
\newblock {\em Disc. Appl. Math.}, 46(3):275--281, 1993.

\bibitem{Brandenburg:1983}
Franz{-}Josef Brandenburg.
\newblock Uniformly growing $k$-th power-free homomorphisms.
\newblock {\em Theor. Comput. Sci.}, 23:69--82, 1983.

\bibitem{Bruyere&Hansel&Michaux&Villemaire:1994}
V.~Bruy\`ere, G.~Hansel, C.~Michaux, and R.~Villemaire.
\newblock Logic and $p$-recognizable sets of integers.
\newblock {\em Bull. Belgian Math. Soc. Simon Stevin}, 1(2):191--238, 1994.
\newblock Corrigendum, 1:577, 1994.

\bibitem{Crochemore:1982}
Maxime Crochemore.
\newblock Sharp characterizations of squarefree morphisms.
\newblock {\em Theor. Comput. Sci.}, 18:221--226, 1982.

\bibitem{Currie&Rampersad:2009}
James~D. Currie and Narad Rampersad.
\newblock There are $k$-uniform cubefree binary morphisms for all $k\ge 0$.
\newblock {\em Disc. Appl. Math.}, 157(11):2548--2551, 2009.

\bibitem{Dejean:1972}
F.~{Dejean}.
\newblock Sur un {th\'eor\`eme} de {Thue}.
\newblock {\em J. Combin. Theory. Ser. A}, 13:90--99, 1972.

\bibitem{Du&Shallit&Shur:2015}
Chen~Fei Du, Jeffrey~O. Shallit, and Arseny~M. Shur.
\newblock Optimal bounds for the similarity density of the {Thue-Morse} word
  with overlap-free and $(7/3)$-power-free infinite binary words.
\newblock {\em Int. J. Found. Comput. Sci.}, 26(8):1147--1166, 2015.

\bibitem{Karhumaki&Shallit:2004}
Juhani Karhum{\"{a}}ki and Jeffrey~O. Shallit.
\newblock Polynomial versus exponential growth in repetition-free binary words.
\newblock {\em J. Comb. Theory, Ser. {A}}, 105(2):335--347, 2004.

\bibitem{Keranen:1984}
Veikko Ker{\"{a}}nen.
\newblock On $k$-repetition freeness of length uniform morphisms over a binary
  alphabet.
\newblock {\em Disc. Appl. Math.}, 9(3):297--300, 1984.

\bibitem{Kobayashi:1986}
Y.~Kobayashi.
\newblock Repetition-free words.
\newblock {\em Theoretical Computer Science}, 44:175--197, 1986.

\bibitem{Leconte:1985}
Michel Leconte.
\newblock A characterization of power-free morphisms.
\newblock {\em Theor. Comput. Sci.}, 38:117--122, 1985.

\bibitem{Meleshko}
J.~Meleshko, P.~Ochem, J.~Shallit, and S.~L. Shan.
\newblock Pseudoperiodic words and a question of {Shevelev}.
\newblock {\em Disc. Math. \& Theor. Comput. Sci.}, 25:Paper \#6, 2023.

\bibitem{Mousavi:2016}
H.~Mousavi.
\newblock Automatic theorem proving in {Walnut}.
\newblock Preprint, available at \url{https://arxiv.org/abs/1603.06017}, 2016.

\bibitem{Petrova&Shur:2012}
Elena~A. Petrova and Arseny~M. Shur.
\newblock Constructing premaximal ternary square-free words of any level.
\newblock In Branislav Rovan, Vladimiro Sassone, and Peter Widmayer, editors,
  {\em Mathematical Foundations of Computer Science 2012---37th International
  Symposium, {MFCS} 2012, Proceedings}, volume 7464 of {\em Lecture Notes in
  Computer Science}, pages 752--763. Springer, 2012.

\bibitem{Richomme&Seebold:2002b}
Gw{\'{e}}na{\"{e}}l Richomme and Patrice S{\'{e}}{\'{e}}bold.
\newblock Conjectures and results on morphisms generating $k$-power-free words.
\newblock {\em Int. J. Found. Comput. Sci.}, 15(2):307--316, 2004.

\bibitem{Richomme&Wlazinski:2002a}
Gw{\'{e}}na{\"{e}}l Richomme and Francis Wlazinski.
\newblock Some results on $k$-power-free morphisms.
\newblock {\em Theor. Comput. Sci.}, 273(1-2):119--142, 2002.

\bibitem{Richomme&Wlazinski:2004}
Gw{\'{e}}na{\"{e}}l Richomme and Francis Wlazinski.
\newblock Overlap-free morphisms and finite test-sets.
\newblock {\em Disc. Appl. Math.}, 143(1-3):92--109, 2004.

\bibitem{Richomme&Wlazinski:2007}
Gw{\'{e}}na{\"{e}}l Richomme and Francis Wlazinski.
\newblock Existence of finite test-sets for $k$-power-freeness of uniform
  morphisms.
\newblock {\em Disc. Appl. Math.}, 155(15):2001--2016, 2007.

\bibitem{Shallit:2011}
J.~Shallit.
\newblock Fife's theorem revisited.
\newblock In G.~Mauri and A.~Leporati, editors, {\em DLT '11: Proceedings of
  the 15th International Conf. on Developments in Language Theory}, volume 6795
  of {\em Lecture Notes in Computer Science}, pages 397--405. Springer, 2011.

\bibitem{Shallit:2022}
Jeffrey Shallit.
\newblock {\em The Logical Approach to Automatic Sequences: Exploring
  Combinatorics on Words with Walnut}.
\newblock London Mathematical Society Lecture Note Series. Cambridge University
  Press, 2022.

\bibitem{Shallit&Shur:2019}
Jeffrey~O. Shallit and Arseny~M. Shur.
\newblock Subword complexity and power avoidance.
\newblock {\em Theor. Comput. Sci.}, 792:96--116, 2019.

\bibitem{Shur:2000}
A.~M. Shur.
\newblock The structure of the set of cube-free {Z}-words in a two-letter
  alphabet.
\newblock {\em Izvestiya Mathematics}, 64:847--871, 2000.

\bibitem{Thue:1912}
A.~Thue.
\newblock {{\"U}ber} die gegenseitige {Lage} gleicher {Teile} gewisser
  {Zeichenreihen}.
\newblock {\em Norske vid. Selsk. Skr. Mat. Nat. Kl.}, 1:1--67, 1912.

\bibitem{Wlazinski:2001}
Francis Wlazinski.
\newblock A test-set for $k$-power-free binary morphisms.
\newblock {\em {RAIRO} Theor. Informatics Appl.}, 35(5):437--452, 2001.

\bibitem{Wlazinski:2016}
Francis Wlazinski.
\newblock Reduction in non-$(k+1)$-power-free morphisms.
\newblock {\em {RAIRO} Theor. Informatics Appl.}, 50(1):3--20, 2016.

\bibitem{Wlazinski:2017}
Francis Wlazinski.
\newblock A uniform cube-free morphism is $k$-power-free for all integers
  $k\geq 4$.
\newblock {\em {RAIRO} Theor. Informatics Appl.}, 51(4):205--216, 2017.

\end{thebibliography}
\end{document}